\documentclass{amsart}
\usepackage{amsmath,amsfonts,amssymb,amscd,amsthm,color}
\usepackage{framed}
\DeclareMathOperator\End{End}
\DeclareMathOperator\Id{Id}
\DeclareMathOperator\Irr{Irr}
\DeclareMathOperator\Res{Res}
\DeclareMathOperator\Trace{Trace}
\DeclareMathOperator\SL{SL}
\DeclareMathOperator\PGL{PGL}
\DeclareMathOperator\GL{GL}
\newcommand\BC{{\mathbb C}}
\newcommand\BF{{\mathbb F}}
\newcommand\BQ{{\mathbb Q}}
\newcommand\BZ{{\mathbb Z}}
\newcommand\cB{{\mathcal B}}
\newcommand\cH{{\mathcal H}}
\newcommand\cL{{\mathcal L}}
\newcommand\cO{{\mathcal O}}
\newcommand\bB{{\bf B}}
\newcommand\bG{{\bf G}}
\newcommand\bH{{\bf H}}
\newcommand\bL{{\bf L}}
\newcommand\bM{{\bf M}}
\newcommand\bS{{\bf S}}
\newcommand\bT{{\bf T}}
\newcommand\bZ{{\bf Z}}
\newcommand\bs{{\bf s}}
\newcommand\bt{{\bf t}}
\newcommand\bw{{\bf w}}
\newcommand\Qlbar{\overline\BQ_\ell}
\newcommand\uni{\text{uni}}
\newcommand\Cmin{C_{\text{min}}}
\newcommand\scal[3]{{\langle\,#1,#2\,\rangle_{#3}}}
\newcommand\lexp[2]{\kern\scriptspace\vphantom{#2}^{#1}\kern-\scriptspace#2}
\newcommand\card[1]{|#1|}
\newcommand\Chevie{{\tt CHEVIE}}
\newcommand\GAP{{\tt GAP}}
\newtheorem{theorem}[equation]{Theorem}
\newtheorem{lemma}[equation]{Lemma}
\def\begchevie{\begin{framed}\noindent}
\def\endchevie{\end{framed}}

\begin{document}
\title {The development version of the \Chevie\ package of \GAP3}

\author{Jean~Michel}
\address{UFR de Math\'ematiques,
Universit\'e Denis Diderot - Paris 7,
Bat. Sophie Germain, case 7012 -- 75013 Paris Cedex 13, France.}
\email{jmichel@math.jussieu.fr}
\maketitle
\section{Introduction}

The  published version 3 of the \Chevie\ package was released with the last
version  of \GAP3 in  1997. The paper  \cite{chevie} documents the state of
the  package in  1994; at  the time  the main  authors of  the package were
M.~Geck, G.~Hiss, F.~L\"ubeck, G.~Malle, and G.~Pfeiffer. The author started
working  on \Chevie\  in 1995,  motivated by  collaboration with M.~Geck on
determining  characters of  Iwahori-Hecke algebras.  For that  he developed
programs  to  compute  in  Artin-Tits  braid  groups,  and improved the way
reflection  subgroups  of  Coxeter  groups  were  handled.  Prompted by the
author's  interest in complex reflection  groups, he also implemented basic
routines  dealing with them  and the associated  cyclotomic Hecke algebras.
This is the extent of his work on \Chevie\ at the publication time in 1997.

Around  that time, the author also started a collaboration with F.~ L\"ubeck
to  implement arbitrary  reductive groups  and their  unipotent characters.
This collaborative work lasted until roughly 2000, at which time the author
became  the  main  developer  of  the  \Chevie\  package,  motivated by his
research  themes and  collaborations with  various people,  including among
others  D.~Bessis,  C.~Bonnaf\'e,  M.~Brou\'e,  M.  Chlouveraki,  F.~Digne,
J.~Gonzalez-Meneses, B.~Howlett, G.~Malle, I.~Marin and R.~Rouquier.

The  latest \Chevie~version, described  here, is available  on the author's
webpage  \cite{moi},  together  with  convenient  facilities  to install it
bundled with \GAP3.

Here  is  a  list  of  some  of  the  additional facilities provided by the
\Chevie\ package since 1997:

\begin{itemize}
\item  Affine Weyl groups and general Coxeter groups, and the corresponding
Hecke algebras with their Kazhdan-Lusztig bases and polynomials. The author
implemented  this via ``generic'' support  for Coxeter groups, writing code
in  terms  of  a  limited  set  of  primitives  \verb+FirstLeftDescending+,
\verb+LeftDescentSet+,  etc$\ldots$, definable for arbitrary Coxeter groups
whatever  the representation (for instance,  Affine Weyl group elements are
represented  as  matrices,  instead  of  the  permutations  used for finite
Coxeter groups).

\item  Kazhdan-Lusztig polynomials  and bases  for unequal  parameter Hecke
algebras.  Hecke  modules  on  Hecke  algebras  for general Coxeter groups,
including the bases defined by Deodhar and Soergel for these modules.

\item 
Reflection  cosets for  arbitrary complex  reflection groups, together with
their  automatic classification. This includes  the possibility of defining
Coxeter  Cosets corresponding to the ``very twisted'' Ree and Suzuki groups
of  Lie type.  Quite a  few methods  work now  for arbitrary finite complex
reflection  groups and cosets, such as type recognition (decomposition into
a  product  of  recognized  irreducible  groups), so routines for character
tables,  for  instance,  have  become  fast  and  easy  to  read using such
decompositions.

\item 
Complete  lists of representations for Hecke algebras of all finite Coxeter
groups, using Howlett's work for the big exceptional groups.
An almost complete list of representations for cyclotomic Hecke
algebras  for finite complex reflection groups;  most of this work was done
jointly  with G.~Malle, see \cite{MalleMichel}. Currently are missing a few
representations  for $G_{29}$  and most  representations for  groups in the
range  $G_{31}$ to $G_{34}$. Similarly,  there are partial character tables
for these cyclotomic Hecke algebras. The character table for the algebra of
$G_{29}$  is complete; the list of representations of the reflection groups
themselves are complete except for $G_{34}$.

\item
Complete  lists of polynomial invariants  for all finite complex reflection
groups.

\item
General  Garside and locally  Garside monoids. This  includes braid monoids
for  general Coxeter groups,  dual braid monoids  for finite Coxeter groups
and  well-generated  complex  reflection  groups.  There  are algorithms to
determine conjugacy sets and compute centralizers, implementing the work of
\cite{GG}.

\item
Semisimple  elements  of  reductive  groups,  including  the computation of
centralizers,  and  determining  the  list  of  isolated and quasi-isolated
classes.  

\item
Unipotent  characters  for  reductive  groups,  their Lusztig induction and
Lusztig's  Fourier  transform  have  been  implemented,  as  well  as $\cL$
functions  attached to Deligne-Lusztig  varieties. The above  has also been
implemented for ``Spetses'' attached to complex reflection groups.

\item
Unipotent  classes of  reductive groups  (including the  bad characteristic
case)  and the generalized Springer correspondence and Green functions. The
Maple  part of the \Chevie\ package dealing with Green functions has become
obsolete,  since  the  corresponding  computations  can now be handled more
conveniently within \GAP.

\item
Systematic  methods for formatting objects, in order to display them nicely
or export them in TeX form or in Maple form.

\item
Some support for posets.
\end{itemize}

One  may ask why this package was  developed in \GAP3, and not \GAP4, which
is the released version of \GAP\ since 1999; the reason is that the authors
of the package made considerable use of generic programming facilities (the
``type  system'') in \GAP3,  which is incompatible  with \GAP4; the package
represents a considerable investment of programming time, and the author is
not  yet willing to  stop his research  for one year,  which is the minimum
time  which  would  be  needed  to  port  the  package  to  \GAP4. The main
limitation  that \GAP3\ imposes is the limitation of memory to 2 gigabytes,
which  may some day be motivation enough for a port; but this port might as
well  be to  another system  like \verb+sage+  (see \cite{sage}), which can
already  use \Chevie\ through its \GAP3\ interface. We mention \cite{pycox}
where  Geck has  ported to  Python (and  substantially improved,  using new
mathematics)  the  \Chevie\  facilities  for Kazhdan-Lustig cells.

The  rest of this paper  introduces the package by  giving some examples of
its  use in  braid groups  and algebraic  groups. This  covers only a small
amount  of the  available facilities  in \Chevie.  However the  coverage is
detailed  in the sense that we give  complete \Chevie\ code for each of the
considered problems. The code in \Chevie\ has already been used intensively
in proving several important results, for instance in \cite{BM}, \cite{KM},
\ldots

\section{Braid groups}
We  begin by looking  at a conjecture  that Lusztig formulated in \cite{L0}
about  the (possibly  twisted) centralizer  of some  elements in  the braid
group.

Let  $(W,S)$  be  a  finite  Coxeter  system  and let $V$ be its reflection
representation,  a real vector  space on which  the elements of  $S$ act by
reflections.  We consider  the `twisted''  situation where  we are given in
addition  an automorphism $\sigma\in\GL(V)$ normalizing $W$ (this situation
is  motivated  by  the  study  of  non-split  reductive  groups).  Such  an
automorphism  is called  a {\em  diagram automorphism}  since we may choose
$\sigma$  up to  an inner  automorphism such  that it  stabilizes $S$. This
situation  defines a {\em Coxeter coset} $W\sigma$. A {\em conjugacy class}
of  $W\sigma$  is  a  $W$-orbit  for  the  conjugation action, and the {\em
centralizer}   of   $x\sigma\in   W\sigma$   is   the   set  $\{w\in  W\mid
wx\sigma=x\sigma  w\}$. All these  notions generalize straightforwardly the
case where $\sigma=\Id$.

A  {\em standard parabolic  subgroup} is a  subgroup of $W$  generated by a
subset  $J\subset  S$.  A  conjugacy  class  of $W$ (resp. of $W\sigma$) is
called  {\em elliptic} if it does not meet any proper subgroup $W_J$ (resp.
any proper subcoset $W_J\sigma$ where $\sigma(J)=J$).

If the presentation of $W$ as a Coxeter group is
$$W=\langle S\mid s^2=1, 
\underbrace{sts\ldots}_{m_{s,t}}=
  \underbrace{tst\ldots}_{m_{s,t}}\text{ for $s,t\in S$}\rangle$$
then the {\em Artin-Tits braid group} attached to $W$
is defined by the presentation
$B^+=\langle \bS\mid
\underbrace{\bs\bt\bs\ldots}_{m_{s,t}}=
  \underbrace{\bt\bs\bt\ldots}_{m_{s,t}}\text{ for $\bs,\bt\in \bS$}\rangle$,
where  $\bS$ is a  copy of $S$,  whose relations are  called the {\em braid
relations}.  There is an obvious quotient  map $B\to W$ since the relations
of  $B$ are relations  in $W$; Matsumoto's  lemma, stating that two reduced
expressions  for  an  element  of  $W$  can  be related by using only braid
relations,  implies that  there is  a well-defined  section of the quotient
which   maps  a  reduced  expression   $w=s_1\ldots  s_n$  to  the  product
$\bs_1\ldots\bs_n\in B$.

Since  $B$ is generated by  a copy $\bS$ of  $S$, the automorphism $\sigma$
extends  naturally to  $B$. The  following result  was proved  for the Weyl
groups of classical groups in \cite{L0}, and later given a general proof in
\cite{HN}.  In  the  meanwhile  the  author  could check it for exceptional
finite Coxeter groups using \Chevie.

\begin{theorem}\label{centbraid}[Lusztig, He-Nie] 
Let $w$ be an element of minimal length of
an  elliptic conjugacy class of $W$ (resp.  of $W\sigma$). Let $\bw$ be the
lift of $w$ to $B$. Then the map $C_B(\bw)\to C_W(w)$ (resp.
$C_B(\bw\sigma)\to C_W(w\sigma)$) is surjective.
\end{theorem}

Actually, one may conjecture that there is a strong structural relationship
between  the groups $C_B(\bw\sigma)$ and $C_W(w\sigma)$: in most cases, the
second  is  a  complex  reflection  group  and  the  first  should  be  the
corresponding braid group.

We  will show the code  in \Chevie\ to check  theorem \ref{centbraid} for a
Coxeter coset of type $\lexp 2E_6$. We first construct the Coxeter group:

\begin{verbatim}
gap> W:=CoxeterGroup("E",6);;
gap> PrintDiagram(W);
E6      2
        |
1 - 3 - 4 - 5 - 6
\end{verbatim}
In  \GAP, the result of a command which ends with a double semicolon is not
printed.  
The command \verb+PrintDiagram+ shows the numbering the elements of $S$. We
now  specify  the  coset  by  giving  the permutation that the automorphism
$\sigma$ does on the elements of $S$.

\begin{verbatim}
gap> WF:=CoxeterCoset(W,(1,6)(3,5));
2E6
\end{verbatim}

We now need information on the conjugacy classes of the coset. There is
a Chevie function which does that for arbitrary complex reflection groups
or associated reflection cosets.

\begchevie\verb+ChevieClassInfo(W)+

This  function returns a record  containing information about the conjugacy
classes  of the finite reflection group or  coset $W$. If the argument is a
coset   $WF$,  where  $W$   denotes  the  reflection   group  and  $F$  the
automorphism,  the classes  are defined  as the  $W$-orbits on $WF$ for the
conjugation action.

The result is a record which contains among others the following fields:

  \verb+.classtext+:   words  in   the  generators   of  $W$  which  define
  representatives  of the  conjugacy classes.  These representatives are of
  minimal length and `very good'' in the sense of \cite{GM}.

  \verb+.classnames+:  names for the conjugacy classes.

  \verb+.classes+:  sizes of the conjugacy classes
\endchevie

\begin{verbatim}
gap> ChevieClassInfo(CoxeterGroup("A",3));
rec(
  classparams := [ [ [ 1, 1, 1, 1 ] ], [ [ 2, 1, 1 ] ], 
      [ [ 2, 2 ] ], [ [ 3, 1 ] ], [ [ 4 ] ] ],
  classnames := [ "1111", "211", "22", "31", "4" ],
  classtext := [ [  ], [ 1 ], [ 1, 3 ], [ 1, 2 ], [ 1, 3, 2 ] ],
  classes := [ 1, 6, 3, 8, 6 ],
  orders := [ 1, 2, 2, 3, 4 ] )
\end{verbatim}

It  has been proven  by Howlett that  a minimal length  representative of a
non-elliptic  conjugacy class lies actually  in a proper standard parabolic
subgroup.  Thus we can find the indices  of the elliptic classes by testing
if  the  minimal  word  for  a  representative  contains an element of each
$F$-orbit on $S$.
\begin{verbatim}
gap> cl:=ChevieClassInfo(WF).classtext;;
gap> elliptic:=Filtered([1..Length(cl)],i->
>   ForAll([[1,6],[3,5],[2],[4]],I->Intersection(cl[i],I)<>[]));
[ 1, 4, 5, 6, 7, 9, 10, 14, 15 ]
\end{verbatim}

We  now construct the lifts in the braid group of theses representatives by
using  the function \verb+Braid(W)+  which takes as  argument a sequence of
indices in $S$ and returns the corresponding element of the braid group:
\begin{verbatim}
gap> B:=Braid(W);
function ( arg ) ... end
gap> cl:=List(cl{elliptic},B);
[ w0, 124315436543, 1231431543165431, 23423465423456, 142314354231465431, 
  45423145, 4254234565423456, 1254, 123143 ]
\end{verbatim}
The  lift of the longest  element of $W$ is  printed in a particular way as
\verb+w0+  since it is the  {\em Garside element} of  the braid monoid. The
braid monoid is a {\em Garside monoid}, that is a cancellative monoid which
has  a special element called the Garside element such that its set of left
and  righ-divisors coincide,  generate the  monoid, and  form a lattice for
divisibility. We now use the function
\begchevie\verb+CentralizerGenerators(b[,F])+

The  element \verb+b+ should be an element of a Garside group. The function
returns  a list  of generators  of the  centralizers of \verb+b+, using the
algorithm of Gebhardt and Gonzalez-Meneses.

If  an  argument  \verb+F+  is  given  it  should  be the automorphism of a
reflection  coset attached to the same group to which the Garside monoid is
attached.  Then  the  \verb+F+-centralizer  is  computed,  defined  as  the
elements $\verb+x+$ such that $x b=b F(x)$.
\endchevie

In  the above, the  ``automorphism of a  reflection coset'' can be obtained
for  a coset  \verb+WF+ by  the call  \verb+Frobenius(WF)+, which returns a
\GAP\  function which  knows how  to apply  the automorphism $F$ to various
objects attached to $W$: words, elements, braids$\ldots$.
\begin{verbatim}
gap> F:=Frobenius(WF);
function ( arg ) ... end
gap> cc:=List(cl,x->CentralizerGenerators(x,F));
[ [ 6, 5, 4, 2, 3, 1 ], [ 65, 124315436543, (3)^-1.24315436543, 
    (4)^-1.1243654, (43)^-1.243654, 13, (5)^-1.12431543654, 
    (35)^-1.2431543654, (45)^-1.124354, (435)^-1.24354 ],
  [ 6, 5, 2431543654, 3, 1 ], 
  [ 4, 2, 342542345, 34265423145, (4354265431)^-1.314354265431, 
      (542345)^-1.65423145 ], 
  [ (5)^-1.1435, (456)^-1.154234565, 142314356, (4356)^-1.5423456 ], 
  [ 4, 13454231435426 ], 
  [ 42, 45426542314354265431, 5423, (2)^-1.5426542314354265431, 
      (3)^-1.56542314354265431, (3)^-1.543, (43)^-1.6542314354265431 ], 
  [ 12542346 ], [ 123465, 123142354654 ] ]
\end{verbatim}
In  the  above,  elements  of  the  braid  group  are  printed as ``reduced
fractions'' \verb+(a)^-1.b+ where \verb+a+ and \verb+b+ are elements of the
braid monoid which have no common left divisor. Some of the generating sets
can be simplified:
\begin{verbatim}
gap> cc{[2,4,5,6]}:=List(cc{[2,4,5,6]},ShrinkGarsideGeneratingSet);
[ [ 65, 13, (4)^-1.1243654 ], [ 4, 2, 342542345, 34265423145 ], 
  [ (5)^-1.1435, 142314356 ], [ 42, 5423, (3)^-1.56542314354265431 ]]
\end{verbatim}

It  is  now  straightforward  to  finish  the  computation.  To check that
$C_B(\bw\sigma)\to  C_W(w\sigma)$ is surjective, we compute the size of the
image, using the function \verb+EltBraid+ which computes the quotient $B\to
W$:
\begin{verbatim}
gap> List(cc,x->Size(Subgroup(W,List(x,EltBraid))));
[ 51840, 648, 216, 108, 96, 10, 72, 9, 12 ]
\end{verbatim}
And  we compare  with the  size of  $C_W(w\sigma)$ that  we can compute two
ways: using the field \verb+.classes+ of \verb+ChevieClassInfo+
\begin{verbatim}
gap> List(ChevieClassInfo(WF).classes{elliptic},x->Size(W)/x);
[ 51840, 648, 216, 108, 96, 10, 72, 9, 12 ]
\end{verbatim}
or asking directly for the size of the centralizer:
\begin{verbatim}
gap> List(ChevieClassInfo(WF).classtext{elliptic},
>   x->Size(Centralizer(W,EltWord(WF,x))));
[ 51840, 648, 216, 108, 96, 10, 72, 9, 12 ]
\end{verbatim}

\section{Representing reductive groups} 

We now describe how to work with reductive groups in \Chevie. We first look
at the case of connected groups; a connected reductive group $\bG$ over an
algebraically  closed field  $\BF$ is  determined up  to isomorphism by the
{\em  root  datum}  $(X(\bT),\Phi,  Y(\bT),\Phi^\vee)$  where  $\Phi\subset
X(\bT)$  are  the  roots  with  respect  to  the  maximal  torus  $\bT$ and
$\Phi^\vee\subset  Y(\bT)$ are  the corresponding  coroots. This determines
the Weyl group, a finite reflection group $W\subset\mathrm{GL}(Y(\bT))$.

In  \Chevie, to specify $\bG$,  we give an integral  matrix $R$ whose lines
represent the simple roots in terms of a basis of $X(\bT)$, and an integral
matrix  $R^\vee$ whose  lines represent  the simple  coroots in  terms of a
basis  of $Y(\bT)$. It is  assumed that the bases  of $X(\bT)$ and $Y(\bT)$
are  chosen  such  that  the  canonical  pairing  is  given  by $\langle x,
y\rangle_\bT=\sum_i x_i y_i$.

For convenience, two particular cases are implemented in \Chevie\ where the
user  just has to specify  the Coxeter type of  the Weyl group. If $\bG$ is
adjoint  then $R$ is the identity matrix  and $R^\vee$ is the Cartan matrix
of  the root system  given by $\{\alpha^\vee(\beta)\}_{\alpha,\beta}$ where
$\alpha^\vee$  (resp. $\beta$) runs  over the simple  coroots (resp. simple
roots).  If  $\bG$  is  semisimple  simply  connected,  then the dual group
$\bG^*$ is adjoint thus the situation is reversed: $R^\vee$ is the identity
matrix  and $R$  the transposed  of the  Cartan matrix.  In all  cases, the
function  we  use  constructs  a  particular  integral  representation of a
Coxeter group, so it is called \verb+CoxeterGroup+.

By default, the adjoint group is returned.
For instance, the group $\mathrm{PGL}_3$ is obtained by

\begin{verbatim}
gap> PGL:=CoxeterGroup("A",2);
CoxeterGroup("A",2)
gap> PGL.simpleRoots;
[ [ 1, 0 ], [ 0, 1 ] ]
gap> PGL.simpleCoroots;
[ [ 2, -1 ], [ -1, 2 ] ]
\end{verbatim}

To get the semisimple simply connected group $\SL_3$, the additional parameter
\verb+"sc"+ has to be given:

\begin{verbatim}
gap> SL:=CoxeterGroup("A",2,"sc");
CoxeterGroup("A",2,"sc")
gap> SL.simpleRoots;              
[ [ 2, -1 ], [ -1, 2 ] ]
gap> SL.simpleCoroots;            
[ [ 1, 0 ], [ 0, 1 ] ]
\end{verbatim}

To get $\mathrm{GL}_3$ we must use the general form by giving $R$ and $R^\vee$:

\begin{verbatim}
gap> GL := CoxeterGroup( [ [ -1, 1, 0], [ 0, -1, 1 ] ],
> [ [ -1, 1, 0], [ 0, -1, 1 ] ] );;
[ [ -1, 1, 0 ], [ 0, -1, 1 ] ]
gap> GL.simpleCoroots;
[ [ -1, 1, 0 ], [ 0, -1, 1 ] ]
\end{verbatim}

\begchevie{\verb+RootDatum(type[,index])+}

For   convenience,  there   is  also   a  function  \verb+RootDatum+  which
understands some familiar names for algebraic groups (like
\verb+"halfspin"+)  and does  the appropriate  call to \verb+CoxeterGroup+.
\endchevie
For instance, the above call is equivalent to

\begin{verbatim}
gap> GL := RootDatum("gl",3);
\end{verbatim}

\section{Computations with semisimple elements using \Chevie}

We  present  briefly  some  of  the  \Chevie~facilities  for computing with
semisimple  elements of finite order in reductive groups, with the programs
for checking some lemmas which were used in \cite{BM}.

Let  $\bS$ be  a torus  defined over  $\BF$. The map $\BF^\times\otimes_\BZ
Y(\bS)\to   \bS$  given   by  $x\otimes\lambda\mapsto   \lambda(x)$  is  an
isomorphism, where we identify $\bS$ to the group of its points over $\BF$.
If  $\BF=\BC$ we may  choose an isomorphism  between the elements of finite
order  in $\BF^\times$ (the roots of unity)  and $\BQ/\BZ$. If $\BF$ is an
algebraic closure of the finite field $\BF_p$, we may choose an isomorphism
$\BF^\times\simeq   (\BQ/\BZ)_{p'}$.  In   these  cases   we  get  thus  an
isomorphism  between $(\BQ/\BZ)\otimes_\BZ Y(\bS)$ and the points of finite
order  of  $\bS$  (resp.  $(\BQ/\BZ)_{p'}\otimes_\BZ Y(\bS)\xrightarrow\sim
\bS$).  Thus, if $\dim \bS=r$, an element of $\bS$ can be represented by an
element of $(\BQ/\BZ)^r$ as soon as we choose a basis of $Y(\bS)$.
If  $\bS$ is  a subtorus  of $\bT$,  then the  inclusion $\bS\subset\bT$ is
determined by giving a basis of the sublattice $Y(\bS)$ inside $Y(\bT)$.

These are the basic ideas used to represent semisimple elements in \Chevie. 

We  recall that a semisimple element $s$ of $\bG$ is {\em isolated} (resp.
{\em  quasi-isolated}) if $C_\bG(s)^0$ (resp. $C_\bG(s)$) does not lie in a
proper  Levi subgroup of $\bG$. We now  show how to use the \Chevie~package
to check the following lemma:

\begin{lemma}\label{calcul}
Let  $\bG$ be an adjoint  group of type $E_6$  and $\bM$ a Levi subgroup of
type  $A_2 \times A_2$.  If $s$ is  a semisimple element  of $\bM$ which is
quasi-isolated  in $\bM$  and in  $\bG$, there  exists $z  \in \bZ(\bM)$ of
order $3$ such that $s$ and $sz$ are not conjugate in $\bG$.
\end{lemma}

\begin{proof}
We  first compute the list  of elements of order  3 of $\bZ(\bM)$. 
The first thing is to specify the Levi subgroup $\bM$.
\begin{verbatim}
gap> G:=CoxeterGroup("E",6);;PrintDiagram(G);
E6      2
        |
1 - 3 - 4 - 5 - 6
gap>M:=ReflectionSubgroup(G,[1,3,5,6]);
ReflectionSubgroup(CoxeterGroup("E",6), [ 1, 3, 5, 6 ])
\end{verbatim}
We now compute the torus $Z(\bM)^\circ=Z(\bM)$.

\begchevie\verb+AlgebraicCentre(G)+

This  function returns  a description  of the  centre $Z$  of the algebraic
group \verb+G+ (it may be a non-connected group, represented as an {\em extended
Coxeter group}, see below) as a record with the following fields\:

\verb+.Z0+:  A  basis  of  $Y(Z^0)$  (with  respect  to  the  canonical basis of
$Y(\bT)$)

\verb+.complement+:  A  basis  of  $Y(\bS)$,  a  complement  lattice to \verb+.Z0+ in
$Y(\bT)$ where $\bS$ is a complement torus to $Z^0$ in $\bT$.

\verb+.AZ+: representatives of $A(Z)\:=Z/Z^0$ given as a subgroup of $\bS$ (that
is, elements of $\BQ/\BZ\otimes Y(\bS)$).
\endchevie

\begin{verbatim}
gap> ZM:=AlgebraicCentre(M).Z0;
[ [ 0, 1, 0, -1, 0, 0 ], [ 0, 0, 0, 1, 0, 0 ] ]
\end{verbatim}

We now ask for the subgroup of elements of order 3 of $Z(\bM)^\circ$.

\begchevie\verb+SemisimpleSubgroup( G, V, n)+

Assuming  that  the  characteristic  of  $\BF$  does  not  divide $n$, this
function returns the subgroup of elements of order dividing \verb+n+ in the
subtorus  $\bS$ of  the maximal  torus $\bT$  of the algebraic group $\bG$,
where $\bS$ is represented by \verb+V+, an integral basis of the sublattice
$Y(\bS)$ of $Y(\bT)$.
\endchevie

\begin{verbatim}
gap> Z3:=SemisimpleSubgroup(G,ZM,3);
Group( <0,1/3,0,2/3,0,0>, <0,0,0,1/3,0,0> )
\end{verbatim}

The  above illustrates how semisimple elements are printed. The group \verb+Z3+
is  represented as  a subgroup  of $\bT$;  elements of  $\bT$, which  is of
dimension  6, are represented as lists of  6 elements of $\BQ/\BZ$ in angle
brackets; elements of $\BQ/\BZ$ are themselves represented as fractions $r$
such   that  $0\le  r  <1$.  The  subgroup   of  elements  of  order  3  of
$Z(\bM)^\circ$ is generated by 2 elements which are given above. We may ask
for the list of all elements of this group.

\begin{verbatim}
gap> Z3:=Elements(Z3);
[ <0,0,0,0,0,0>, <0,0,0,1/3,0,0>, <0,0,0,2/3,0,0>, 
  <0,1/3,0,2/3,0,0>, <0,1/3,0,0,0,0>, <0,1/3,0,1/3,0,0>, 
  <0,2/3,0,1/3,0,0>, <0,2/3,0,2/3,0,0>, <0,2/3,0,0,0,0> ]
\end{verbatim}
We now compute the list of elements quasi-isolated in both $\bG$ and $\bM$.
\begin{verbatim}
gap> reps:=QuasiIsolatedRepresentatives(G);
[ <0,0,0,0,0,0>, <0,0,0,0,1/2,0>, <0,0,0,1/3,0,0>, 
  <0,1/6,1/6,0,1/6,0>, <1/3,0,0,0,0,1/3> ]
\end{verbatim}
The  list  \verb+reps+  now  contains  representatives of $\bG$-orbits of
quasi-isolated elements. The algorithm to get these was described in
\cite{bonnafe quasi}. To get all the
quasi-isolated elements in $\bT$, we need to take the orbits under the Weyl
group:
\begin{verbatim}
gap> qi:=List(reps,s->Orbit(G,s));;
gap> List(qi,Length);
[ 1, 36, 80, 1080, 90 ]
\end{verbatim}
We have not displayed the orbits since they are quite large: the first orbit
is that of the identity element, which is trivial, but the fourth contains
$1080$ elements. We now filter
each orbit by the condition to be quasi-isolated also in $\bM$.
\begin{verbatim}
gap> qi:=List(qi,x->Filtered(x,y->IsQuasiIsolated(M,y)));;
gap> List(qi,Length);
[ 1, 3, 26, 36, 12 ]
gap> qi[2];
[ <0,0,0,1/2,0,0>, <0,1/2,0,1/2,0,0>, <0,1/2,0,0,0,0> ]
\end{verbatim}
There  is a way to  do the same computation  which does not need to compute
the large intermediate orbits under the Weyl group of $\bG$. The idea is to
compute  first the orbit of a semisimple quasi-isolated representative $s$
under  representatives  of  the  double cosets $C_\bG(s)\backslash\bG/\bM$,
which  are not too many,  then test for being  quasi-isolated in $\bM$, and
finally  take the orbits  under the Weyl  group of $\bM$.  

We use the following \Chevie\ function:

\begchevie{\verb+SemisimpleCentralizer( G, s)+}

This  function returns the  stabilizer in the  Weyl group of the semisimple
element \verb+s+ of the algebraic group \verb+G+. The result describes also
$C_\bG(s)$  since it is returned as  an extended reflection group, with the
reflection  group part  equal to  the Weyl  group of  $C_\bG^0(s)$, and the
diagram  automorphism part  being that  induced by $C_\bG(s)/C_\bG^0(s)$ on
$C_\bG^0(s)$.   The  extended  reflection  groups  represent  non-connected
reductive groups which are semi-direct product of their connected component
by a group of diagram automorphisms.
\endchevie

So starting with \verb+reps+ as above, we first compute:
\begin{verbatim}
ce:=List(reps,s->SemisimpleCentralizer(G,s));;ce[5];
Extended(ReflectionSubgroup(CoxeterGroup("E",6), 
[ 2, 3, 4, 5 ]),<(2,5,3)>)
gap> ce[5].group;
ReflectionSubgroup(CoxeterGroup("E",6), [ 2, 3, 4, 5 ])
gap> ce[5].permauts;
Group( ( 1,72, 6)( 2, 5, 3)( 7,71,11)( 8,10, 9)(12,70,16)
(13,14,15)(17,68,21)(18,69,20)(22,66,25)(23,67,65)(26,63,28)
(27,64,62)(29,59,31)(30,61,58)(32,57,53)(33,56,54)(34,52,48)
(35,47,43)(36,42,37)(38,41,39)(44,46,45)(49,50,51) )
\end{verbatim}
We show for the 5th element of \verb+reps+ how an extended reflection group
is  represented: it  contains a  reflection subgroup  of the  Weyl group of
$\bG$,   the   Weyl   group   of   $C_\bG^\circ(s)$,   obtained   above  as
\verb+ce[5].group+,  extended by the group of diagram automorphisms induced
on  it  by  $C_\bG(s)$,  obtained  above  as  \verb+ce[5].permauts+;  these
automorphisms  are  denoted  by  the  permutation  of  the  simple roots of
$C_\bG^\circ(s)$ they induce.

To get the whole Weyl group of $C_\bG(s)$ we need to combine these two pieces.
For this we define a \GAP\ function:
\begin{verbatim}
TotalGroup:=g->Subgroup(G,Concatenation(g.group.generators,
g.permauts.generators));
\end{verbatim}
We then compute representatives of the double cosets 
$C_\bG(s)\backslash\bG/\bM$, we apply them to \verb+reps+, keep
the ones still quasi-simple in $\bM$:
\begin{verbatim}
dreps:=List(ce,g->List(DoubleCosets(G,TotalGroup(g),M),
Representative));;
qi:=List([1..Length(reps)],i->List(dreps[i],w->reps[i]^w));;
qi:=List(qi,x->Filtered(x,y->IsQuasiIsolated(M,y)));
[ [ <0,0,0,0,0,0> ], [ <0,1/2,0,0,0,0>, <0,0,0,1/2,0,0>, 
  <0,1/2,0,1/2,0,0> ], [ <0,0,0,1/3,0,0>, <0,0,0,2/3,0,0>, 
  <1/3,2/3,1/3,0,2/3,2/3>, <1/3,1/3,1/3,0,2/3,2/3> ], 
  [ <1/3,1/2,1/3,1/2,2/3,2/3>, <1/3,1/2,1/3,0,2/3,2/3>, 
  <2/3,0,2/3,5/6,2/3,2/3> ], [ <1/3,0,1/3,0,1/3,1/3> ] ]
\end{verbatim}
We get a list such that the $\bM$-orbits of the sublists give the same list
as  before. We will  need this previous  list of all $\bG$-conjugates which
are  $\bM$-quasi-isolated, so if we did not  keep it we recompute this list
containing the $\bM$-orbits of the sublists by
\begin{verbatim}
qim:=List(qi,l->Union(List(l,s->Orbit(M,s))));;
\end{verbatim}
We  now ask, for each element $s$ of  each of our orbits, how many elements
$z$  of \verb+Z3+ are such  that $s$ and $sz$  are not $\bG$-conjugate. The
test  for being conjugate is that $sz$  is in the same $\bG$-orbit. We need
to make the test only for our representatives of the $\bM$-orbits, since if
$s$  is $\bG$-conjugate to  $sz$ with $z\in\bZ(\bM)$,  then $m s m^{-1}$ is
$\bG$-conjugate to $m s m^{-1}z=m sz m^{-1}$.
\begin{verbatim}
gap> List([1..Length(qi)],i->List(qi[i],s->Number(Z3,
  z->PositionProperty(qim,o->s*z in o)<>i)));
[ [ 8 ], [ 8, 8, 8 ], [ 7, 7, 3, 3 ], [ 6, 6, 6 ], [ 6 ] ]
\end{verbatim}
and we find indeed that there is always more than $0$ elements $z$ which work.
Note that the function \verb+PositionProperty+ returns \verb+false+ when no
element is found satisfying the given property, thus the number counted is
the $z$ such that $s$ and $sz$ are in a different orbit, as well as the cases
when $sz$ is not quasi-isolated in $\bG$.
\end{proof}

\section{Rational structures}
We now assume that $\BF$ is an algebraic closure of $\BF_p$ and that
$F$ is the Frobenius on $\bG$ corresponding to an $\BF_q$-structure where 
$q$ is a power of $p$. We assume that $\bG$ is an adjoint group
of type $E_7$ which in \Chevie\ has the following labelling of the simple
roots:
\begin{verbatim}
gap> G:=CoxeterGroup("E",7);;PrintDiagram(G);
E7      2
        |
1 - 3 - 4 - 5 - 6 - 7
\end{verbatim}
We are going to show the \Chevie\ code for the following lemma. 

\begin{lemma}\label{ordre 8}
Assume  $\bM$ is  an $F$-stable  Levi of  type $A_1  \times A_1 \times A_1$
corresponding to the roots 2, 3, 5 in the above diagram. If $q \in \{3,5\}$
and if $\card{\bZ(\bM)^F}=\Phi_1(q)^a\Phi_2(q)^b$ with $a$, $b \ge
1$, then $\bZ(\bM)^F$ contains an element of order $8$.
\end{lemma}

In  \Chevie, to specify an $\BF_q$-structure  on a reductive group, we give
an  element $\phi\in\mathrm{GL}(Y(\bT))$ such that $F=q\phi$. We may choose
$\phi$  such that it  stabilizes the set  of simple roots.  Such an element
$\phi$ is determined by the coset $W\phi\subset\mathrm{GL}(Y(\bT))$, so the
structure which represents it in \Chevie\ is a Coxeter coset.

Further,  if $\bM'$ is  an $F$-stable $\bG$-conjugate  of the Levi subgroup
$\bM$,  the pair $(\bM',F)$  is isomorphic to  $(\bM,wF)$ for some $w\in W$
(determined  by $\bM'$ up  to  $F$-conjugacy).  So,  given  a Coxeter coset
$W\phi$,  an $F$-stable conjugate of a Levi  subgroup whose Weyl group is a
standard  parabolic subgroup $W_I$ is represented by a subcoset of the form
$W_I w\phi$, where $w\phi$ normalizes $W_I$.

To  check the lemma,  we first compute  the list of  elements of order 8 of
$\bZ(\bM)$, using the same commands as shown before.
\begin{verbatim}
gap> M:=ReflectionSubgroup(G,[2,5,7]);;
gap> ZM:=AlgebraicCentre(M);;
gap> Z8:=SemisimpleSubgroup(G,ZM.Z0,8);
Group( <1/8,0,0,0,0,0,0>, <0,0,1/8,7/8,0,1/8,0>, 
<0,0,0,1/8,0,7/8,0>, <0,0,0,0,0,1/8,0> )
gap> Z8:=Elements(Z8);;Length(Z8);
4096
\end{verbatim}
We  now ask for representatives of the $\bG^F$-classes of $F$-stable
$\bG$-conjugates   of   $\bM$. The group $\bG$ is split, so $\phi$ is trivial.
Thus an $F$-stable-conjugate of $\bM$ is represented by a coset of the
form $W_I w$. We first ask for  the list of all possible such
twistings of $\bM$:
\begin{verbatim}
gap> Mtwists:=Twistings(G,M);
[ A1<2>xA1<5>xA1<7>.(q-1)^4, 
  (A1xA1xA1)<2,5,7>.(q-1)^2*(q^2+q+1), 
  A1<2>xA1<5>xA1<7>.(q-1)^2*(q^2+q+1), 
  (A1xA1xA1)<2,5,7>.(q^2+q+1)^2, 
  (A1xA1xA1)<2,7,5>.(q-1)*(q+1)*(q^2+q+1), 
  (A1xA1xA1)<2,7,5>.(q-1)*(q+1)*(q^2-q+1), 
  ...
\end{verbatim}
In the above list (of 24 entries of which the first 6 are listed), brackets
around  pairs or triples of  $A_1$ denote an orbit  of the Frobenius on the
components. The element $w$ is not displayed, but the order
$\card{\bZ(\bM)^{wF}}$ is displayed. We want to keep the sublist where that
order  is a  product of  $\Phi_1(q)$ and  $\Phi_2(q)$. For  this we use the
function

\begchevie\verb+PhiFactors(WF)+

Let  \verb+WF+ be  a reflection  coset of  the form  $W\phi$, and let $V$ be the
vector  space on which $W$ acts as a reflection group. Let $f_1,\ldots,f_n$
be  the basic invariants of $W$ on  the symmetric algebra of $V$, chosen so
that  $\phi$ has the $f_i$  as eigenvectors. The corresponding eigenvalues,
listed  in the  same order  as \verb+ReflectionDegrees(W)+  (the degrees  of the
$f_i$) are called the {\em factors} of $\phi$ acting on $V$.
\endchevie

\begin{verbatim}
gap> Mtwists:=Filtered(Mtwists,MF->Set(PhiFactors(MF))=[-1,1]);
[ A1<2>xA1<5>xA1<7>.(q+1)^4, 
  A1<2>xA1<5>xA1<7>.(q-1)^2*(q+1)^2, 
  (A1xA1)<2,7>xA1<5>.(q-1)^3*(q+1), 
  A1<2>xA1<5>xA1<7>.(q-1)*(q+1)^3, 
  A1<2>xA1<5>xA1<7>.(q-1)^3*(q+1), 
  (A1xA1)<2,7>xA1<5>.(q-1)*(q+1)^3, 
  (A1xA1)<2,7>xA1<5>.(q-1)^2*(q+1)^2 ]
\end{verbatim}
Here  \verb+PhiFactors+ gives the  eigenvalues of $w$  on the invariants of
the Weyl group of $\bM$ acting on the symmetric algebra of
$X(\bT)\otimes\BC$.  The cases  we want  is when  these eigenvalues are all
equal  to  $1$  or  $-1$  (actually  this  gives  us  one extra case, where
$\card{\bZ(\bM)^{wF}}=(q+1)^4$   since  the   eigenvalues  on  the
complement of $\bZ(\bM)$ are always $1$; we will just have to disregard the
first entry of \verb+Mtwists+).

Now  for each of the remaining \verb+Mtwists+ we compute the fixed points of
$wF$ on \verb+Z8+, and look at the maximal order of an element in there. We
first  illustrate the necessary commands one by one on an example before 
showing a line of code which combines them.
\begin{verbatim}
gap> Z8F:=Filtered(Z8,s->Frobenius(Mtwists[3])(s)^3=s);
[ <0,0,0,0,0,0,0>, <0,0,0,1/4,0,1/4,0>, 
  <0,0,0,1/2,0,1/2,0>, <0,0,0,3/4,0,3/4,0>, 
  ...
  <3/4,0,3/4,1/8,0,3/8,0>, <3/4,0,3/4,3/8,0,5/8,0>, 
  <3/4,0,3/4,5/8,0,7/8,0>, <3/4,0,3/4,7/8,0,1/8,0> ]
\end{verbatim}
The   expression  \verb+Frobenius(Mtwists[3])+  returns  a  function  which
applies  to  its  argument  the  $w\phi$  associated  to the third twisting
\verb|(A1xA1)<2,7>xA1<5>.(q-1)^3*(q+1)|  of $\bM$. To compute $wF$ we still
have  to raise to the third power since  $q=3$. We give above 8 of the
32 entries obtained; we can see from the denominators that some elements in
the  resulting list of $wF$-stable elements of \verb+Z8+ are of order 8. We
can make this easier to see by writing a small function:
\begin{verbatim}
gap> OrderSemisimple:=s->Lcm(List(s.v,Denominator));
gap> List(Z8F,OrderSemisimple);
[ 1, 4, 2, 4, 2, 4, 2, 4, 8, 8, 8, 8, 8, 8, 8, 8, 
  2, 4, 2, 4, 2, 4, 2, 4, 8, 8, 8, 8, 8, 8, 8, 8 ]
gap> Set(last);
[ 1, 2, 4, 8 ]
\end{verbatim}
We now do the computation for all cosets in one command:
\begin{verbatim}
gap> List(Mtwists,MF->Set(List(Filtered(Z8,s->Frobenius(MF)(s)^3=s),
> OrderSemisimple)));
[ [ 1, 2, 4 ], [ 1, 2, 4, 8 ], [ 1, 2, 4, 8 ], 
  [ 1, 2, 4, 8 ], [ 1, 2, 4, 8 ], [ 1, 2, 4, 8 ], 
  [ 1, 2, 4, 8 ] ]
\end{verbatim}
and  we  see  that  indeed,  apart  from  the  first  twist which should be
disregarded,  for all  twists the  fixed points  of \verb+Z8+ still contain
elements of order 8.

\section{Lusztig's map from conjugacy classes in the Weyl group to conjugacy
classes in the reductive group}
Let $\bG$ be a connected reductive group over an algebraically closed field
$\BF$ of characteristic $p\ge 0$, and let $\cB$ be the variety of its Borel
subgroups. The $\bG$-orbits (for the diagonal action) on $\cB\times\cB$ are
naturally  indexed by the  Weyl group $W$  of $\bG$; we  denote $\cO_w$ the
orbit indexed by $w\in W$.

For   $\gamma$  a   conjugacy  class   of  $\bG$   and  $w\in  W$,  we  set
$$\cB_w^\gamma:=\{(g,\bB)\in\gamma\times\cB\mid (\bB,\lexp
g\bB)\in\cO_w\}.$$

For  $C$ a conjugacy class of $W$  and $\gamma$ a unipotent class of $\bG$,
Lusztig  denotes  $\gamma\vdash  C$  if  $\cB_w^\gamma\ne\emptyset$ for any
$w\in \Cmin$ where $\Cmin$ is the set of elements of $C$ of minimal length.

 In \cite{L1} Lusztig shows:
\begin{theorem}\label{cW to cG}If $p$ is good for $\bG$ then
\begin{itemize}
\item For $C$ a class of $W$, among the classes such that $\gamma\vdash C$,
there  exists  a  unique  class  $\gamma_C$  minimal  for the partial order
defined  by:
$$\text{$\gamma\le\gamma'$ if and only  if $\gamma$ lies in the Zariski
closure of $\gamma'$.}$$
\item The map $C\mapsto\gamma_C$ is surjective.
\end{itemize}
\end{theorem}

Lusztig  uses  \Chevie\  to  show  Theorem  \ref{cW  to cG} for exceptional
groups.  He was  not using  the development  version; see  also \cite{geck}
describing  the same computation.  The problem is  very easy to solve using
the  development version of \Chevie, as we will show by giving the complete
code  to solve it; Lusztig  notes in an addendum  in \cite{L1} that Theorem
\ref{cW  to cG} still  holds in bad  characteristic for exceptional groups.
Our code will also check this.

The  idea  of  the  computation  is  as  follows.  Assume that $\BF$ is an
algebraic closure of a finite prime field $\BF_p$. Let $F$ be the Frobenius
corresponding to a split rational structure of $\bG$ over $\BF_q$ where $q$
is  a power of  $p$. Lusztig gives  a formula for $\card{(\cB_w^\gamma)^F}$
when $\gamma$ is unipotent, which shows that $\card{(\cB_w^\gamma)^F}$ is a
polynomial in $q$. It is thus equivalent that $\cB_w^\gamma\ne\emptyset$ or
$\card{(\cB_w^\gamma)^F}\ne  0$ as  a polynomial,  which enables us to do the
computation since this polynomial can be readily computed in \Chevie.

To  give the formula, we need some more notation related to the permutation
module  $\Qlbar\cB^F$ for $\bG^F$  (where $\ell$ is  a prime different from
$p$;  we use $\Qlbar$ instead of $\BC$ as coefficients to make a connection
with  $\ell$-adic  cohomology,  see  below).  The Hecke algebra, defined as
$\cH:=\End_{\bG^F}\Qlbar\cB^F$  has  basis  $\{T_w\}_{w\in  W}$  defined by
$T_w(\bB)=\sum_{\{\bB'\mid  (\bB,\bB')\in\cO_w^F\}}\bB'$. The algebra $\cH$
specializes for $q\mapsto 1$ to $\Qlbar W$, inducing a bijection $E_q\mapsto E:
\Irr(\cH)\to\Irr(W)$   and   as   $\cH\times\bG^F$-module   we   have   the
decomposition    $\Qlbar\cB^F=\sum_{E\in\Irr    W}E_q\otimes\rho_E$   where
$\rho_E$  is an irreducible {\em  unipotent representation of the principal
series} of $\bG^F$.

Finally,  for $E$ a  representation of $W$,  let us define  the {\em almost
character} $R_E^\bG:=\card W^{-1}\sum_{w\in W}\Trace(w\mid
E)R_{\bT_w}^\bG(\Id)$  where $\bT_w$ is an  $F$-stable maximal torus of $\bG$
of  type  $w$  and  $R_{\bT_w}^\bG$  is Deligne-Lusztig induction; here
$R_{\bT_w}^\bG(\Id)$ is a virtual $\Qlbar\bG^F$-module.
Lusztig's formula is
$$\card{(\cB_w^\gamma)^F}=\sum_{E'\in\Irr(W)}
\Trace(T_w\mid \sum_{E\in\Irr W}\scal{\rho_E}{R_{E'}^\bG}{\bG^F}E_q)
\sum_{u\in\gamma^F}\Trace(u\mid R_{E'}^\bG).\eqno(a)$$
By  \cite[Theorem 1.1 (b)]{GP} the conjugacy class of $T_w$ in $\cH$ is the
same for all $w\in\Cmin$ thus by (a) the condition
$\cB_w^\gamma\ne\emptyset$  does not  depend on  the choice of $w\in\Cmin$;
this  makes  the  computation  doable  in  a  group like $W(E_8)$ which has
696729600 elements but only 112 conjugacy classes.

Now, we introduce some \Chevie\ functions which
can be used to compute the right-hand side of (a).
We decompose the computation in two steps, computing the functions 
$$ f(C,E'):=
\Trace(T_w\mid \sum_{E\in\Irr W}\scal{\rho_E}{R_{E'}^\bG}{\bG^F}E_q)
\text{ for $w\in\Cmin$,} $$
and $g(E',\gamma):=\sum_{u\in\gamma^F}\Trace(u\mid R_{E'}^\bG)$.

To compute $f$ we need:
\begchevie{\verb+UnipotentCharacters(G)+}

This  function returns a record  containing information about the unipotent
characters of the reductive group \verb+G+. Some important fields are:

\verb+charNames+: the list of names of the unipotent characters.

\verb+harishChandra+:  information  about  Harish-Chandra  series  of  unipotent
characters.   This  is  itself  a  list  of  records,  one  for  each  pair
$(\bL,\lambda)$  of  a  Levi  of  an  $F$-stable  parabolic  subgroup and a
cuspidal unipotent character of $\bL^F$.

\verb+families+:  information  about  Lusztig  families of unipotent characters.
These  families  correspond  to  blocks  of  the  matrix of scalar products
$\scal\rho{R_E^\bG}{\bG^F}$   where   $\rho$   runs   over   the  unipotent
characters. This matrix can also be obtained from the \verb+UnipotentCharacters+
record.
\endchevie

As an example we now show the result of TeXing the output of\hfill\break
\verb+FormatTeX(UnipotentCharacters(CoxeterGroup("G",2)));+

\medskip
Unipotent characters for $G_2$
{\tabskip 0mm\halign{\strut\vrule\kern 1mm\hfill$#$
\kern 1mm\vrule&\kern 1mm\hfill$#$\kern 1mm&\kern 1mm\hfill$#$
\kern 1mm&\kern 1mm\hfill$#$\kern 1mm&\kern 1mm\hfill$#$\kern 1mm\vrule\cr
\noalign{\hrule}
\gamma&\hbox{Deg($\gamma$)}&\hbox{FakeDegree}&\hbox{Fr($\gamma$)}&\hbox{Label}\cr
\noalign{\hrule}
\phi_{1,0}&1&1&1&\cr
\phi_{1,6}&q^6&q^6&1&\cr
\phi_{1,3}'&{1\over3}q\Phi_3\Phi_6&q^3&1&(1,\rho)\cr
\phi_{1,3}''&{1\over3}q\Phi_3\Phi_6&q^3&1&(g_3,1)\cr
\phi_{2,1}&{1\over6}q\Phi_2^2\Phi_3&q\Phi_8&1&(1,1)\cr
\phi_{2,2}&{1\over2}q\Phi_2^2\Phi_6&q^2\Phi_4&1&(g_2,1)\cr
G_2[-1]&{1\over2}q\Phi_1^2\Phi_3&0&-1&(g_2,\varepsilon)\cr
G_2[1]&{1\over6}q\Phi_1^2\Phi_6&0&1&(1,\varepsilon)\cr
G_2[\zeta_3]&{1\over3}q\Phi_1^2\Phi_2^2&0&\zeta_3&(g_3,\zeta_3)\cr
G_2[\zeta_3^2]&{1\over3}q\Phi_1^2\Phi_2^2&0&\zeta_3^2&(g_3,\zeta_3^2)\cr
\noalign{\hrule}}
}

\medskip
The  \verb+FakeDegree+  is  defined  for  a unipotent character $\rho_E$ as
$R_E^\bG(1)$  and  is  $0$  outside  of  the  principal  series. The column
$\hbox{Fr}(\gamma)$  contains  the  root  of  unity  part  of the Frobenius
eigenvalue  attached  to  the  character  $\gamma$  when  it appears in the
cohomology  of a Deligne-Lusztig  variety. Finally the  ``Label'' refers to
the labelling by Lusztig families. 

To compute $f$ we use also:

\begchevie{\verb+DeligneLusztigLefschetz(h)+}

By  \cite[III,  1.3  and  2.3]{DM},  the  class  function  on $\bG^F$ which
associates  to $g\in\bG^F$  the number  of fixed  points of  $g F^m$ on the
Deligne-Lusztig  variety associated  to the  element $w\in  W$ has, for $m$
sufficiently    divisible,    the    form    $g\mapsto   \sum_{E\in\Irr(W)}
\Trace(T_w\mid  E_{q^m}) R_E^\bG(g)$.  This expression  is called  the {\em
Lefschetz character} associated to the Deligne-Lusztig variety; since $\cH$
splits  over  $\BQ[\sqrt  q]$,  it  is  a  sum of unipotent characters with
coefficients in $\BQ[\sqrt q]$.

The  function \verb+DeligneLusztigLefschetz+  takes as  argument a Hecke element
\verb+h+ and returns the  corresponding Lefschetz character (the definition
is extended from $T_w$ to any \verb+h+$\in\cH$ by linearity).
\endchevie

Note that, since
$\scal{\rho_E}{R_{E'}^\bG}{\bG^F}=\scal{\rho_{E'}}{R_E^\bG}{\bG^F}$
(see  \cite[III, 3.5(iii)]{DM}), the function  $f(C,E')$ is the coefficient
of  \verb+DeligneLusztigLefschetz+$(T_w)$  on  $\rho_{E'}$.  

Here  is a function which  takes as arguments $W$  and $q$, and returns the
matrix of $f(C,E')$, with rows indexed by $C$ and columns indexed by $E'$.

\begin{verbatim}
f:=function(W,q)
  return List(ChevieClassInfo(W).classtext,
    function(w)local Tw,psuc;
      Tw:=Basis(Hecke(W,q),"T")(w);
      psuc:=UnipotentCharacters(W).harishChandra[1].charNumbers;
      return DeligneLusztigLefschetz(Tw).v{psuc};
    end);
end;
\end{verbatim}

We   use   here   that   \verb+ChevieClassInfo(W).classtext+   contains   a
representative of $C$ in $\Cmin$. The variable \verb+psuc+ which stands for
``principal   series  unipotent  characters''  holds  the  indices  of  the
unipotent  characters  of  the  principal  series  $\rho_E$ amongst all the
unipotent   characters.  The   principal  series   is  the   first  of  the
Harish-Chandra series described by the list
\verb+UnipotentCharacters(W).harishChandra+.  The  field  \verb+v+  of  the
\verb+DeligneLusztigLefschetz(Tw)+ record is a list of coefficients on each
unipotent character.

The  function \verb+f+  takes on  a 2Ghz  computer 3  seconds for  $E_7$, and 23
seconds for $E_8$.

To  write  the  code  for  the  function  $g$, we will need to describe the
\Chevie\   functions   dealing   with   unipotent  elements,  the  Springer
correspondence, and the Green functions.

Let  $(\gamma,\varphi)$ run  over the  pairs where  $\gamma$ is a unipotent
class  of $\bG$, and  $\varphi$ is a  character of the  group of components
$A(u):=C_\bG(u)/C_\bG^0(u)$  of  $C_\bG(u)$  for  $u\in\gamma$.  Such pairs
describe  $\bG$-equivariant local  systems on  $\gamma$. The  {\em Springer
correspondence}  is  an  injective  map  from  $\Irr(W)$ to the set of such
pairs.  It is not  surjective in general,  but all the pairs $(\gamma,\Id)$
are  in the  image. The  {\em generalized  Springer Correspondence} extends
this to a bijection. This time the source is the union of
$\Irr(W_\bG(\bL))$,  where  $\bL$  runs  over  Levi  subgroups  (taken up to
conjugacy)  which admit a {\em cuspidal} local system $(\lambda,\phi)$; the
group  $W_\bG(\bL):=N_\bG(\bL)/\bL=N_\bG(\bL,\phi)/\bL$ is  a Coxeter group
for  such Levis. The image of each of  these sets is called a {\em Springer
series}.  We will  just need  the ordinary  Springer correspondence  in our
computation  (whose image is  called the {\em  principal Springer series}),
but the \Chevie\ functions describe the generalized correspondence.

\begchevie{\verb+UnipotentClasses(G[,p])+}

This  function returns a record  containing information about the unipotent
classes of the algebraic group \verb+G+ in characteristic \verb+p+ (if omitted, \verb+p+ is
assumed  to be any good characteristic for $\bG$). Some important fields of
the record are:

\verb+.orderClasses+:  a list describing the partial order on unipotent classes.
The  poset is described by its Hasse diagram, that is the \verb+i+-th element of
the  list is the list  of the indices \verb+j+  of the classes immediately above
the  \verb+i+-th class. That  is \verb+.orderclasses[i]+ contains  \verb+j+ if ${\overline
\gamma}_j\supsetneq  \gamma_i$ and there  is no class  $\gamma_k$ such that
${\overline  \gamma}_j\supsetneq {\overline \gamma}_k \supsetneq \gamma_i$.

\verb+.classes+: a list of records holding information for each unipotent class.
In particular the field \verb+.Au+ holds the group $A(u)$.

\verb+.springerSeries+:  a list of  records, each of  which describes a Springer
series  of $\bG$. The main  field in such a  record is \verb+.locsys+, a list of
length  \verb+NrConjugacyClasses(WGL)+,  where  \verb+WGL+  is the group $W_\bG(\bL)$
associated  to the  series, holding  in \verb+i+-th  position a  pair describing
which  local system corresponds to the \verb+i+-th character of \verb+WGL+. The first
element  of the pair is the index of the concerned unipotent class $u$, and
the second is the index of the corresponding character of $A(u)$.

The  record returned by \verb+UnipotentClasses+  depends on  the isogeny  type of
$\bG$ and on the characteristic. We give some examples.

First, here is what gives TeXing the output of
\hfill\break \verb+FormatTeX(UnipotentClasses(CoxeterGroup("G",2)))+:

\noindent$1{<}A_1{<}\tilde A_1{<}G_2(a_1){<}G_2$\hfill\break
{\tabskip 0mm\halign{\strut\vrule\kern 1mm\hfill$#$
\kern 1mm\vrule&\kern 1mm\hfill$#$\kern 1mm&\kern 1mm\hfill$#$
\kern 1mm&\kern 1mm\hfill$#$\kern 1mm&\kern 1mm\hfill$#$
\kern 1mm&\kern 1mm\hfill$#$\kern 1mm\vrule\cr
\noalign{\hrule}
u&\hbox{diagram}&\dim{\mathcal B}_u&A(u)&G_2()&.(G_2)\cr
\noalign{\hrule}
G_2&22&0&.&\phi_{1,0}&\cr
G_2(a_1)&20&1&A_2&21:\phi_{1,3}'\kern 0.8em 3:\phi_{2,1}&111:\cr
\tilde A_1&01&2&.&\phi_{2,2}&\cr
A_1&10&3&.&\phi_{1,3}''&\cr
1&00&6&.&\phi_{1,6}&\cr
\noalign{\hrule}}
}
\medskip

The  first line  describes the  partial order  of the unipotent classes. In
that  line  and  the  first  column  the  traditional name of the unipotent
classes  are  used.  The  columns  ``diagram''  shows the Dynkin-Richardson
diagram  which we do not explain here.  The next column shows the dimension
of  the  variety  $\cB_u:=\{\bB\in\cB\mid  \bB\owns  u\}$. The next columns
describe   the  Springer  correspondence.  Only  the  second  class  has  a
non-trivial $A(u)$, equal to the Coxeter group of type $A_2$ (the symmetric
group  on 3 elements); this group has three irreducible characters, indexed
by  the partitions $3$, $21$ and $111$. The first two characters are in the
ordinary   Springer  correspondence,   corresponding  to   the  irreducible
characters  $\phi'_{1,3}$ and $\phi_{2,1}$ of  $G_2$. The third corresponds
to  a cuspidal local system, so is a Springer series by itself. The head of
a  column  describing  a  Springer  series  attached  to  the cuspidal pair
$(\bL,(\lambda,\phi))$  is of the form \verb+A(B)+ where \verb+A+ describes
$W_\bG(\bL)$ and \verb+B+ describes $\bL$.

The result is different in characteristic three
\hfill\break \verb+FormatTeX(UnipotentClasses(CoxeterGroup("G",2),3))+:

\noindent$1{<}A_1,(\tilde A_1)_3{<}\tilde A_1{<}G_2(a_1){<}G_2$\hfill\break
{\tabskip 0mm\halign{\strut\vrule\kern 1mm\hfill$#$
\kern 1mm\vrule&\kern 1mm\hfill$#$\kern 1mm&\kern 1mm\hfill$#$
\kern 1mm&\kern 1mm\hfill$#$\kern 1mm&\kern 1mm\hfill$#$
\kern 1mm&\kern 1mm\hfill$#$\kern 1mm&\kern 1mm\hfill$#$\kern 1mm\vrule\cr
\noalign{\hrule}
u&\dim{\mathcal B}_u&A(u)&G_2()&.(G_2)&.(G_2)&.(G_2)\cr
\noalign{\hrule}
G_2&0&Z_{3}&1:\phi_{1,0}&&\zeta_3:&\zeta_3^2:\cr
G_2(a_1)&1&A_1&2:\phi_{2,1}&11:&&\cr
\tilde A_1&2&.&\phi_{2,2}&&&\cr
A_1&3&.&\phi_{1,3}''&&&\cr
(\tilde A_1)_3&3&.&\phi_{1,3}'&&&\cr
1&6&.&\phi_{1,6}&&&\cr
\noalign{\hrule}}
}
\endchevie

As another example we show the difference between $\PGL_n$ and $\SL_n$:
\hfill\break \verb+FormatTeX(UnipotentClasses(RootDatum("pgl",4)))+:

\noindent$1111{<}211{<}22{<}31{<}4$\hfill\break
{\tabskip 0mm\halign{\strut\vrule\kern 1mm\hfill$#$
\kern 1mm\vrule&\kern 1mm\hfill$#$\kern 1mm&\kern 1mm\hfill$#$
\kern 1mm&\kern 1mm\hfill$#$\kern 1mm&\kern 1mm\hfill$#$\kern 1mm\vrule\cr
\noalign{\hrule}
u&\hbox{diagram}&\dim{\mathcal B}_u&A(u)&A_3()\cr
\noalign{\hrule}
4&222&0&.&4\cr
31&202&1&.&31\cr
22&020&2&.&22\cr
211&101&3&.&211\cr
1111&000&6&.&1111\cr
\noalign{\hrule}}
}
\medskip

\verb+FormatTeX(UnipotentClasses(RootDatum("sl",4)))+:

\noindent$1111{<}211{<}22{<}31{<}4$\hfill\break
{\tabskip 0mm\halign{\strut\vrule\kern 1mm\hfill$#$\kern 1mm\vrule&
\kern 1mm\hfill$#$\kern 1mm&\kern 1mm\hfill$#$\kern 1mm&\kern 1mm\hfill$#$
\kern 1mm&\kern 1mm\hfill$#$\kern 1mm&\kern 1mm\hfill$#$\kern 1mm&
\kern 1mm\hfill$#$\kern 1mm&\kern 1mm\hfill$#$\kern 1mm\vrule\cr
\noalign{\hrule}
u&\hbox{diagram}&\dim{\mathcal B}_u&A(u)&A_3()&A_1(A_1^2)/-1&.(A_3)/i&
.(A_3)/-i\cr
\noalign{\hrule}
4&222&0&Z_{4}&1:4&-1:2&i:&-i:\cr
31&202&1&.&31&&&\cr
22&020&2&A_1&2:22&11:11&&\cr
211&101&3&.&211&&&\cr
1111&000&6&.&1111&&&\cr
\noalign{\hrule}}
}
\medskip
In  $\SL_4$ there  are more  local systems  on the  same classes.  Some are
cuspidal,  and  some  are  in  a  Springer  series  corresponding to a Levi
subgroup  of type $A_1\times A_1$. In  the column heads for Springer series,
we  see a third parameter after a \verb+/+ describing a character of $Z\bG$
attached to the series.

We need one more function to write the code for $g$:
\begchevie{\verb+ICCTable(uc[,seriesNo[,q]])+}

The  first  argument  is  a  record  describing  the unipotent classes of a
reductive  group  in  some  characteristic.  \verb+ICCTable+  gives the table of
decompositions  of  the  functions  $  X_{\gamma,\varphi}$  in terms of the
functions   $Y_{\gamma,\varphi}$,   where   $Y_{\gamma,\varphi}$   is   the
characteristic   function  of  the   local  system  $(\gamma,\varphi)$  and
$X_{\gamma,\varphi}$  is the  characteristic function  of the corresponding
intersection cohomology complex.

Since  the coefficient of $X_{\gamma,\varphi}$ on $Y_{\gamma',\varphi'}$ is
$0$  if  $(\gamma,\varphi)$  and  $(\gamma',\varphi')$  are not in the same
Springer  series, the table given is for  a single Springer series, the one
whose  number is given by the argument \verb+seriesNo+ (if omitted this defaults
to   \verb+seriesNo=1+  which  is  the   principal  series).  The  decomposition
multiplicities  are graded,  and are  given as  polynomials in one variable
(specified  by the argument \verb+q+; if not given \verb+Indeterminate(Rationals)+ is
assumed).

The function \verb+ICCTable+ returns a record with various pieces of information
which can help further computations. Some important fields are:

\verb+.scalar+:  the main result, the table of multiplicities of the $X_\psi$ on
the $Y_\chi$, where both $\psi$ and $\chi$ run over $\Irr(W_\bG(\bL))$.

\verb+.dimBu+:  The  list  of  $\dim\cB_u$  for  each local system in the chosen
Springer series.

\verb+.L+:  the matrix of unnormalized scalar products of the functions $Y_\psi$
with  themselves, that is, if  $F$ is a Frobenius  corresponding to a split
$\BF_q$-structure   on  $\bG$,   the  $(\phi,\psi)$   entry  is   equal  to
$\sum_{g\in\bG^F}   Y_\phi(g)   \overline{Y_\psi(g)}$.   This   is  thus  a
symmetric,  block-diagonal matrix  where the  diagonal blocks correspond to
geometric unipotent conjugacy classes.
\endchevie

The  relationship of \verb+ICCTable(uc)+  with characters of  $\bG^F$ is as
follows:  let $X_{(\gamma,\varphi)}$ be a  characteristic function as above
where  $(\gamma,\varphi)$ is in the ordinary Springer correspondence, image
of  $E'\in\Irr(W)$; we  also write  $X_{E'}$ for $X_{(\gamma,\varphi)}$ and
write  $b_{E'}$  for  $\dim\cB_u$.  Then  the  restriction  of  the  almost
character  $R_{E'}^\bG$ to the  unipotent elements is  equal to $q^{b_{E'}}
X_{E'}$. Thus
$g(E',\gamma)=q^{b_{E'}}\sum_{g\in\bG^F}X_{E'}(g)Y_{(\gamma,\Id)}(g)$, thus
if   $(\gamma,\Id)$   is   parameterized   by   $\psi\in\Irr(W)$,  we  have
$g(E',\gamma)=  q^{d_{E'}}P_{E',\psi}L_\psi$  where  $P_{E',\psi}$  is  the
coefficient of $X_{E'}$ on $Y_\psi$ and $L_\psi=
\sum_{g\in\bG^F}   Y_\psi(g)   \overline{Y_\psi(g)}$.

Here is a \Chevie\ function which takes as arguments $W$, the indeterminate
$q$   and  the  characteristic  $p$  and  computes  the  matrix  of  values
$g(E',\gamma)$,  with rows indexed by $E'$ and columns indexed by unipotent
classes $\gamma$.

\begin{verbatim}
g:=function(W,q,p)local uc,t,triv,d;
  uc:=UnipotentClasses(W,p);
  t:=ICCTable(uc,1,q);
  triv:=List([1..Size(uc)],i->
    Position(uc.springerSeries[1].locsys,
      [i,PositionId(uc.classes[i].Au)]));
  d:=List(t.dimBu,i->q^i);
  t:=Zip(t.scalar{triv},DiagonalOfMat(t.L{triv}{triv}),
    function(a,b)return a*b;end);
  return Zip(TransposedMat(t),d,function(a,b)return a*b;end);
end;
\end{verbatim}

The  list \verb+triv+ records, for each  unipotent conjugacy class, the position
in the ordinary Springer correspondence (that is, in the list of irreducible
characters of the Weyl group) of the trivial local system in that class. As is
traditional in functional languages, \verb+Zip(l,v,f)+ takes as arguments two
lists \verb+l+ and \verb+v+ and a function \verb+f+ and returns a list
\verb+z+ such that \verb+z[i]=f(l[i],v[i])+.

We put what we have done so far together and we get:

\begin{verbatim}
vdash:=function(W,p)local q;q:=X(Rationals)^2;
  return List(f(W,q)*g(W,q,p),x->Filtered([1..Length(x)],
    i->x[i]<>0*X(Rationals)));
end;
\end{verbatim}

The  function  \verb+vdash+  returns  a  list  indexed  by  the  classes of
\verb+W+,  which  contains  for  the  \verb+i+-th  class  $C_i$ the list of
indices $j$ of unipotent classes $\gamma_j$ such that $\gamma_j\vdash C_i$.
Note  that we use \verb+X(Rationals)^2+ as  a variable since to compute the
character  table of the Hecke  algebras of types $E_7$  and $E_8$ (which is
used  by \verb+DeligneLusztigLefschetz+)  \Chevie\ must  be able to extract
square roots of the indeterminate.

It  remains to check Lusztig's  theorem \ref{cW to cG},  and to compute the
map  $C\mapsto\gamma_C$. The following  function returns a  list indexed by
the conjugacy classes of \verb+W+ whose element indexed by $C$ is the index
of $\gamma_C$.

\begin{verbatim}
gamma:=function(W,p)local lt;
  lt:=Incidence(Poset(UnipotentClasses(W,p)));
  return List(vdash(W,p), function(l)local classes;
    classes:=Filtered(l,x->ForAll(l,y->lt[x][y]));
    if Length(classes)<>1 then Error("no minimal class \gamma_C");fi;
    return classes[1];
  end);
end;
\end{verbatim}

We  have used \Chevie\ functions for posets: \verb+Poset(UnipotentClasses(W,p))+
returns  the poset  defined by  Zariski closure  of unipotent  classes as a
\Chevie\  object,  and  \verb+Incidence+  returns  the  incidence matrix of this
poset,  that  is,  \verb+lt[i][j]+  is  \verb+true+  if  and only if $\gamma_i\subset
\overline\gamma_j$.  

The following function finally checks that $C\mapsto\gamma_C$ is surjective
and displays this map.

\begin{verbatim}
LusztigMap:=function(W,p)local g,uc,i;
  g:=gamma(W,p);
  uc:=UnipotentClasses(W,p);
  if Set(g)<>[1..Size(uc)] then Error("not surjective");fi;
  for i in [1..Size(uc)] do
    Print(Join(List(Positions(g,i),j->ClassName(W,j)),", "),
          " -> ",ClassName(uc,i),"\n");
  od;
end;
\end{verbatim}

Here is the result for $G_2$ in characteristics $0$ and $3$:

\begin{verbatim}
gap> LusztigMap(CoxeterGroup("G",2),0);
A0 -> 1
A1 -> A1
~A1, A1+~A1 -> ~A1
A2 -> G2(a1)
G2 -> G2
gap> LusztigMap(CoxeterGroup("G",2),3);
A0 -> 1
A1 -> A1
A1+~A1 -> ~A1
A2 -> G2(a1)
G2 -> G2
~A1 -> (~A1)3
\end{verbatim}

The  analogous computation takes 2 minutes on  a 2 Ghz computer for a given
characteristic in type $E_8$.

\section{A generalization}
In  \cite{L2},  Lusztig  generalizes  the  above  results to the case where
$\gamma$ is not necessarily unipotent. We now assume $\bG$ simply connected
to  ensure that the  centralizers of semisimple  elements are connected. 

We will see in formula (b) below that the condition
$\cB_w^\gamma\ne\emptyset$   is   still   independent   of  the  choice  of
$w\in\Cmin$.  For a given class $C\subset W$, Lusztig denotes $\delta_C$ the
minimum  dimension of a class $\gamma$ such that $\cB_w^\gamma\ne\emptyset$
for  $w\in\Cmin$. It is clear from theorem \ref{cW to cG} that $\delta_C\le
\dim \gamma_C$. Finally Lusztig denotes
$$\fbox{$\bG_C$}=\bigcup_{\{\gamma\mid \cB_w^\gamma\ne\emptyset\text{ and }
\dim\gamma=\delta_C\}} \gamma$$

The main result of \cite{L2} is
\begin{theorem}\label{l2}
\begin{itemize}
\item  Given two classes $C$  and $C'$ of $W$,  the sets \fbox{$\bG_C$} and
\fbox{$\bG_{C'}$} are equal or disjoint.
\item \fbox{$\bG_C$} are the pieces of a stratification of $\bG$.
\end{itemize}
\end{theorem}

Assume again that $\BF$ is an algebraic closure of $\BF_p$, and that $F$ is
the Frobenius associated to an $\BF_q$-structure for $q$ a power of $p$. Up
to  replacing $F$ by some  power we may assume  that $\gamma$ is $F$-stable
and  that there exists $s$, the semisimple part of an $F$-stable element of
$\gamma$, such that the restriction of $F$ to $\bH:=C_\bG(s)$ is split.

Then, with the notation $f(C,E')$ introduced after formula (a)
we have for $w\in\Cmin$
$$
\card{(\cB_w^\gamma)^F}=\frac{\card{\bG^F}}{\card{\bH^F}}\sum_{E'\in\Irr(W)}
f(C,E')\sum_{u\in \bH^F_\uni\mid su\in\gamma}
\Trace(u\mid R_{\Res^W_{W_\bH}E'}^\bH).
\eqno(b)$$

It  is  clear  from  formula  (b) that $\card{(\cB_w^\gamma)^F}$ depends on
$\gamma$  only through  $\bH$ and  the unipotent class $\gamma_u\subset\bH$
defined as the unipotent part of $\gamma$. Also, using the same argument as
for  (a),  the  right-hand  side  of  (b)  does not depend on the choice of
$w\in\Cmin$.

To  compute (b),  we proceed  as for  (a). The  main difference  is that we
compute  the function $g$ in  the group $\bH$, and  multiply on the left by
the  matrix describing the  restriction of characters  from $W$ to $W_\bH$.
Here is the \Chevie\ code replacing the \verb+vdash+ function:

\begin{verbatim}
  List(f(W,q)*InductionTable(H,W).scalar*g(H,q,p),
    x->Filtered([1..Length(x)],i->x[i]<>0*X(Rationals)));
\end{verbatim}

To  check  theorem  \ref{l2},  we  need  a  list  of possible groups $\bH$.
According  to results  of \cite{carter}  and \cite{deriziotis}, to describe
their  Weyl  groups  we  must  consider  up to $W$-conjugacy the reflection
subgroups of $W$ generated by subsets of the {\em extended Dynkin diagram},
formed,  for each irreducible component of $W$, by the simple roots and the
negative  of  the  highest  root;  for  a  given characteristic $p$ we must
exclude  such  subgroups  which  have  a  coefficient divisible by $p$ when
expressing  their  simple  roots  in  an  adapted basis of the initial root
system.

The following function does the job for an irreducible $W$:

\begin{verbatim}
SemisimpleCentralizerRepresentatives:=function(W,p)
  local cent,E,J,R,indices;
  if Length(ReflectionType(W))<>1 then 
    Error("only implemented for irreducible groups");
  fi;
  indices:=W->W.rootInclusion{W.generatingReflections};
  cent:=[];
  E:=Concatenation(indices(W),[W.rootInclusion[2*W.N]]);
  for J in Combinations(E) do
    R:=ReflectionSubgroup(W,J);
    if ForAll(cent,G->IsomorphismType(R)<>IsomorphismType(G) or
      RepresentativeOperation(W,indices(R),indices(G),OnSets)=false)
    then Add(cent,R);
    fi;
  od;
  if p=0 then return cent;fi;
  return Filtered(cent,
    G->ForAll(Concatenation(SmithNormalFormMat(W.roots{indices(G)})),
      x->x=0 or x mod p<>0));
end;
\end{verbatim}

The  function  \verb+IsomorphismType+  describes the  isomorphism type of a
reflection group via, in the Weyl group case, a character string describing
the  isomorphism  type  of  the  root  system; for instance \verb-"A2+~A2"-
describes  the union  of two  systems of  type $A_2$  where the  second one
consists of short roots.

Such  a simple  program is  enough to  detect some  errors in the tables of
\cite{deriziotis}:  for $\bG$ of  type $F_4$ the  type $A_2+\tilde A_2$ for
$\bH$  is excluded in  characteristic 3, not  2; the same  is true for type
$\bH=3A_2$ in $\bG=E_7$; finally the type $\bH=A_3+3A_1$ (which is excluded
in characteristic 2) is altogether forgotten in $\bG=E_8$.

We now have all that we need to compute and display the strata.

\begin{verbatim}
LusztigMapb:=function(W,p)local l,i,s,cent,m,q,fmat,strata,map,cover;
  q:=X(Rationals)^2; fmat:=f(W,q);
  cent:=SemisimpleCentralizerRepresentatives(W,p);
  l:=List([1..Length(cent)],function(i)local H,t,uc;
    H:=cent[i];uc:=UnipotentClasses(H,p);
    t:=List(fmat*InductionTable(H,W).scalar*g(H,q,p),
       x->Filtered([1..Length(x)],i->x[i]<>0*x[i]));
    return List(t, x->List(x,cl->rec(H:=i,
                       class:=ClassName(uc,cl),
		       dim:=2*(W.N-uc.classes[cl].dimBu))));
    end);
  strata:=List(TransposedMat(l),function(s)local mindim;
    s:=Concatenation(s); mindim:=Minimum(List(s,x->x.dim));
    return Filtered(s,x->x.dim=mindim);end);
  map:=CollectBy([1..NrConjugacyClasses(W)],strata);
  strata:=Set(strata);
  if ForAny([1..Length(strata)],i->ForAny([i+1..Length(strata)],j->
    Length(Intersection(strata[i],strata[j]))<>0)) then
    Error("strata are not disjoint\n");
  fi;
  cover:=CollectBy(Concatenation(strata),i->i.H);
  if ForAny([1..Length(cover)],
    i->Length(cover[i])<>Size(UnipotentClasses(cent[i],p))) then
    Error("strata do not cover");
  fi;
  for i in [1..Length(map)] do
    Cut(SPrint("class(es) ",Join(List(map[i],j->ClassName(W,j)),", "),
       " => stratum(",strata[i][1].dim,") ",
       Join(List(strata[i],
         x->SPrint(IsomorphismType(cent[x.H]),":",x.class))," ")),
      rec(places:=" "));
  od;
end;
\end{verbatim}

The  function \verb+CollectBy(l,f)+ takes as argument a list \verb+l+ and a
function  \verb+f+ and returns a list of lists, each of them collecting all
elements  \verb+l[i]+  such  that  \verb+f(l[i])+  takes a given value. The
second  argument \verb+f+ may also be a list of same length as \verb+l+ and
this  time are collected the \verb+l[i]+  such that \verb+f[l[i]]+ takes a
given  value.  The  result  is  sorted  by  the value taken of \verb+f+.

Here are the results of \verb+LusztigMapb+ for $G_2$ in characteristics 0 and 3:

\begin{verbatim}
gap> LusztigMapb(CoxeterGroup("G",2),0);
class(es) G2 => stratum(12) : A1:2 G2:G2 A2:3 ~A1:2 ~A1+A1:2,2
class(es) A2 => stratum(10) A1:11 G2:G2(a1) A2:21 ~A1:11 
                            ~A1+A1:11,2 ~A1+A1:2,11
class(es) A0 => stratum(0) G2:1
class(es) A1 => stratum(6) G2:A1
class(es) A1+~A1 => stratum(8) G2:~A1 ~A1+A1:11,11
class(es) ~A1 => stratum(6) A2:111

gap> LusztigMapb(CoxeterGroup("G",2),3);
class(es) G2 => stratum(12) : A1:2 G2:G2 ~A1:2 ~A1+A1:2,2
class(es) A2 => stratum(10) A1:11 G2:G2(a1) ~A1:11 ~A1+A1:11,2 
                            ~A1+A1:2,11
class(es) ~A1 => stratum(6) G2:(~A1)3
class(es) A0 => stratum(0) G2:1
class(es) A1 => stratum(6) G2:A1
class(es) A1+~A1 => stratum(8) G2:~A1 ~A1+A1:11,11
\end{verbatim}
The classes $\gamma$ composing a stratum are written in the form \verb+A:B+
where  \verb+A+ describes the  isomorphism type of  the Weyl group of $\bH$
and \verb+b+ describes $\gamma_u$.

We notice that no group $\bH$ is of type $A_2$ in characteristic 3; in both
cases  the class  \verb+~A1+ of  $W$ corresponds  to a  stratum composed of
conjugacy classes of dimension 6, but the classes involved are quite different.

In  the listing obtained for $E_8$ in  characteristic 0, one can notice the
same  strata as pointed out  by Lusztig: in all  characteristics there is a
unipotent stratum

\begin{verbatim}
class(es) A1 => stratum(58) E8:A1
\end{verbatim}

And there is a stratum
\begin{verbatim}
class(es) 8A1, 6A1, 4A1'', 7A1, 5A1 => stratum(128) E8:4A1 
                                            D8:1111111111111111
\end{verbatim}
which loses the semisimple class \verb+D8:1111111111111111+ in characteristic 2.

\section{Spetses}
In  \cite{grade} and \cite{BMM},  quite a few  features of algebraic groups
are  generalized to a large class of finite complex reflection groups called
``Spetsial''.  In particular, there is a definition of a set of ``unipotent
characters''  of the ``associated  reductive group'' (which  does not exist
but  is a set of combinatorial data called a ``Spets''). Here is an example
of  our computations,  to be  compared with  the table  given above for the
unipotent  characters in a  reductive group of  type $G_2$. We consider the
group  $G_4$  in  the  Shephard-Todd  classification, which is a reflection
subgroup of $\GL(\BC^2)$ of size 24, isomorphic to $\SL_2(\BF_3)$.

Here is the result of TeXing the output of
\begin{verbatim}
FormatTeX(UnipotentCharacters(ComplexReflectionGroup(4)));
\end{verbatim}

\hfill{Unipotent characters for $G_{4}$}\hfill\hfill\break
{\tabskip 0mm\halign{\strut\vrule\kern 1mm\hfill$#$\kern 1mm\vrule&\kern 
1mm\hfill$#$\kern 1mm&\kern 1mm\hfill$#$\kern 1mm&\kern 1mm\hfill$#$\kern 1mm&
\kern 1mm\hfill$#$\kern 1mm\vrule\cr
\noalign{\hrule}
\hbox{$\gamma$}&\hbox{Deg($\gamma$)}&\hbox{FakeDegree}&\hbox{Fr($\gamma$)}&
\hbox{Label}\cr
\noalign{\hrule}
\phi_{1,0}&1&1&1&\cr
\phi_{1,4}&{-\sqrt {-3}\over6}q^4{\Phi''_3}\Phi_4{\Phi''_6}&q^4&1&
1\!\wedge\!-\zeta_3^2\cr
\phi_{1,8}&{\sqrt {-3}\over6}q^4{\Phi'_3}\Phi_4{\Phi'_6}&q^8&1&
-1\!\wedge\!\zeta_3^2\cr
\phi_{2,5}&{1\over2}q^4\Phi_2^2\Phi_6&q^5\Phi_4&1&1\!\wedge\!\zeta_3^2\cr
\phi_{2,3}&{3+\sqrt {-3}\over6}q{\Phi''_3}\Phi_4{\Phi'_6}&q^3\Phi_4&1&
1\!\wedge\!\zeta_3^2\cr
\phi_{2,1}&{3-\sqrt {-3}\over6}q{\Phi'_3}\Phi_4{\Phi''_6}&q\Phi_4&1&
1\!\wedge\!\zeta_3\cr
\phi_{3,2}&q^2\Phi_3\Phi_6&q^2\Phi_3\Phi_6&1&\cr
Z_3:2&{-\sqrt {-3}\over3}q\Phi_1\Phi_2\Phi_4&0&\zeta_3^2&
\zeta_3\!\wedge\!\zeta_3^2\cr
Z_3:11&{-\sqrt {-3}\over3}q^4\Phi_1\Phi_2\Phi_4&0&\zeta_3^2&
\zeta_3\!\wedge\!-\zeta_3\cr
G_4&{-1\over2}q^4\Phi_1^2\Phi_3&0&-1&-\zeta_3^2\!\wedge\!-1\cr
\noalign{\hrule}}
}
\medskip
As  in the Weyl  group case, unipotent  characters have a  degree, but this
time  it is not given by a  polynomial with rational coefficients, but with
coefficients  in $\BQ[\sqrt{-3}]$, which is  the smallest subfield of $\BC$
over which $G_4$ can be realized.

\end{document}